\documentclass[preprint,12pt]{elsarticle}

\usepackage{amssymb}
\usepackage{amsmath}
\usepackage{amsthm}

\usepackage{url}
\usepackage{microtype}
\usepackage{caption}
\usepackage{doi}
\usepackage{extarrows}
\usepackage{esint}
\usepackage{enumitem}

\usepackage{cleveref} 

\journal{Nuclear Physics B} 

\newtheorem{theorem}{Theorem}[section]
\newtheorem{definition}[theorem]{Definition}

\newtheorem{lemma}[theorem]{Lemma}
\newtheorem{remark}[theorem]{Remark}
\newtheorem{example}[theorem]{Example}
\newtheorem{proposition}[theorem]{Proposition}

\newcommand{\norm}[1]{\left\|#1\right\|}

\newcommand{\supop}{\operatorname*{sup}}

\newcommand{\minop}{\operatorname*{min}}

\begin{document}
	
	\begin{frontmatter}
		
		
		
		\title{Some inequalities and geometric constants in $p$-normed spaces}
		
		\author[label1]{Zhiyao Fang}
		\author[label1]{Qi Liu\corref{cor1}}  
		\ead{liuq67@aqnu.edu.cn}  
		\author[label1]{Yuxin Wang}
		\author[label2]{Yongjin Li}
		\cortext[cor1]{Corresponding author.}  
		
		\affiliation[label1]{organization={School of Mathematics and Statistics, Anqing Normal University},
			addressline={}, 
			city={Anqing},
			postcode={246133}, 
			state={},
			country={P.R.China}}
		
		\affiliation[label2]{organization={Department of Mathematics, Sun Yat-sen University},
			addressline={}, 
			city={Guangzhou},
			postcode={510275}, 
			state={},
			country={P.R.China}}
		
		\begin{abstract}
			In this paper, we study some geometric constants in complete $p$-normed spaces with $0 < p \leq 1$. We introduce a new symmetric geometric constant associated with isosceles orthogonality, establish its sharp bounds, and provide an orthogonal characterization of the generalized von Neumann-Jordan constant in such spaces. We also investigate two Milman-type moduli in complete $p$-normed spaces, including their fundamental properties and sharp product inequalities. Finally, we extend the relation between the James constant and the generalized von Neumann-Jordan constant .
		\end{abstract}

		\begin{keyword}
			$p$-normed space \sep isosceles orthogonality \sep  geometric constant 
		\end{keyword}

	\end{frontmatter}
	
		
		
	
	\section{Introduction}
	
	The investigation into the geometric properties of Banach spaces represents a crucial domain within the broader discipline of functional analysis. In this theory, the geometric properties of normed spaces often play an important role, such as uniform convexity, uniform non-squareness, etc. Quantifying the geometric constants of the geometric features of normed spaces is very useful for studying geometric properties. For instance, Clarkson \cite{Clarkson1937} introduced the notion of the modulus of convexity to describe uniformly convex spaces, and the von Neumann-Jordan constant to quantify how much a Banach space deviates from a Hilbert space \cite{Clarkson1937}, which is defined as
	\[C_{\rm NJ}(X) = \sup\left\{ \frac{\norm{x+y}^2 + \norm{x-y}^2}{2\left( \norm{x}^2 + \norm{y}^2 \right)} : x,y\in X, (x,y)\neq(0,0) \right\}.\]
	It is well known that $1\leq C_{\rm NJ}(X)\leq2$ for any Banach space $X$, and $C_{\rm NJ}(X)=1$ if and only if $X$ is a Hilbert space \cite{Clarkson1937}. Moreover, a Banach space $X$ is uniformly non-square if and only if $C_{\rm NJ}(X)<2$ \cite{Kato2001}.

	To provide an equivalent characterization of the von Neumann-Jordan constant from the perspective of unit sphere vectors, Yang and Wang \cite{Yang2006} introduced the function $\gamma_X(t)$. For any $t\in[0,1]$, it is defined by
	\[\gamma_X(t) = \sup\left\{ \frac{\norm{x+ty}^2 + \norm{x-ty}^2}{2} : x,y\in S_X \right\}.\]
	This function has many excellent properties: it is non-decreasing, convex and continuous on $[0,1]$, and $X$ is a Hilbert space if and only if $\gamma_X(t)=1+t^2$ for all $t\in[0,1]$ \cite{Yang2006}. Moreover, the von Neumann-Jordan constant can be expressed as $C_{\rm NJ}(X) = \displaystyle\sup_{t\in[0,1]} \frac{\gamma_X(t)}{1+t^2}$ \cite{Yang2006}.

	As a fundamental tool to describe the smoothness of Banach spaces, the modulus of smoothness was also introduced in the classical literature. \cite{Lindenstrauss1963}It is defined by the function
	\[\rho_X(t) = \sup\left\{ \frac{\norm{x+ty} + \norm{x-ty}}{2} - 1 : x,y\in S_X \right\}, \quad t\in[0,+\infty).\]
	A Banach space $X$ is said to be uniformly smooth if $\displaystyle\lim_{t \to 0^+} \frac{\rho_X(t)}{t} = 0$.

	In Euclidean geometry, the concept of orthogonality is indispensable. On one hand, it is manifested in the fourth axiom of Euclidean geometry, and on the other hand, it plays a crucial role in the Pythagorean theorem. However, Banach space geometry is significantly different from Euclidean geometry because there is no unique concept of orthogonality in Banach space geometry. With the development of Banach space geometry, many different orthogonalities have been introduced into general normed linear spaces. For instance, James \cite{James1945} introduced isosceles orthogonality ($\perp_I$) and Pythagorean orthogonality ($\perp_P$):
	\[
	x \perp_I y \ \text{if and only if} \ \|x + y\| = \|x - y\|,
	\]
	and
	\[
	x \perp_P y \ \text{if and only if} \ \|x - y\|^2 = \|x\|^2 + \|y\|^2.
	\]

	In addition to the above orthogonalities, Singer orthogonality \cite{Singer1957} is another widely studied orthogonality type in Banach spaces. For two vectors $x$ and $y$ in a Banach space $X$, $x$ is said to be Singer orthogonal to $y$, denoted by $x \perp_S y$, if $\|x\|\|y\| = 0$ or
	\[
	\left\| \frac{x}{\|x\|} + \frac{y}{\|y\|} \right\| = \left\| \frac{x}{\|x\|} - \frac{y}{\|y\|} \right\|.
	\]
	To quantify the difference between Singer orthogonality and isosceles orthogonality, a geometric constant was introduced in \cite{Bi2026} as follows:
	\[
	SI(X)= \sup \left\{ \frac{\|x+y\|}{\|x-y\|} : x,y \in X, x \perp_S y, (x,y) \neq (0,0) \right\},
	\]
	where $x \perp_S y$ denotes that $x$ is Singer orthogonal to $y$. This constant provides an effective quantitative tool for measuring the deviation between different orthogonalities.
	
	To characterize the uniform non-squareness and normal structure of Banach spaces, Gao and Lau \cite{Gao1990} introduced the James constant as follows:
	\[
	J(X) = \sup\left\{ \min\left\{ \norm{x+y}, \norm{x-y} \right\} : x,y\in S_X \right\}.
	\]
	Correspondingly, Kato et al. \cite{Kato2001} defined the Schäffer constant which is closely related to the James constant:
	\[
	S(X) = \inf\left\{ \max\left\{ \norm{x+y}, \norm{x-y} \right\} : x,y\in S_X \right\}.
	\]
	Xiao and Zhu \cite{Xiao2025} extended both constants to complete $p$-normed spaces($0<p\leq 1$) as
	\[	
	J_p(X) = \sup\left\{ \min\left\{ \norm{x+y}_p, \norm{x-y}_p \right\} : x,y\in S_X \right\},
	\]
	\[	
	S_p(X) = \inf\left\{ \max\left\{ \norm{x+y}_p, \norm{x-y}_p \right\} : x,y\in S_X \right\},
	\]
	and obtained the sharp bounds $2^{p-1} \leq J_p(X) \leq 2$ and $2^{p-1} \leq S_p(X) \leq 2^p$.

	As a tool to describe the geometric properties of complete $p$-normed spaces($0<p\leq 1$), the modulus of $p$-rotundity was introduced in \cite{XiaoYuan2025(2)}. For a complete $p$-normed space $(X,\|\cdot\|_{p})$ with $0<p\leq 1$ , it is defined for $\varepsilon\in(0,2]$ as
	\[
	\begin{aligned}
		\Delta_{X}^{p}(\varepsilon)
		&=\inf\left\{1-\frac{\|x+y\|_{p}^{\frac{1}{p}}}{2^{\frac{1}{p}}} : x,y\in S_{X},\ \|x-y\|_{p}\geq 2^{1-p}\varepsilon^{p}\right\}\\
		&=\inf\left\{1-\frac{\|x+y\|_{p}^{\frac{1}{p}}}{2^{\frac{1}{p}}} : x,y\in S_{X},\ \|x-y\|_{p}=2^{1-p}\varepsilon^{p}\right\}.
	\end{aligned}
	\]

	Later, in order to investigate the distance between isosceles orthogonality and Pythagorean orthogonality, Yang et al. \cite{Yang2022} defined a new orthogonal geometric constant, as follows:
	\[
	\Omega_X(a) = \sup\left\{ \frac{\|ax + y\|^2 + \|x + ay\|^2}{\|x + y\|^2} : x, y \in X, (x,y) \neq (0,0), x \perp_I y \right\}, \ 
	\]\text{where} $\ 0 \leq a < 1. $\\
	They first gave the upper and lower bounds of the constant $\Omega_X(a)$, that is, $1 + a^2$ and $2$, and then established the identity of $\Omega_X(a) = \frac{1+a^2}{2}\gamma_X(\frac{1-a}{1+a})$. Finally, building on this foundational identity, they delineated the relationship between the constant $\Omega_X(a)$ and the intrinsic geometric properties of Banach spaces. This exploration encompasses uniformly non-square, uniformly smooth, uniformly convex, and normal structure.
	
	Most recently, Ni et al. \cite{Ni2025} introduced a symmetric geometric constant based on isosceles orthogonality in Banach spaces, which is the direct prototype of our work. For any $t\in[0,\frac{1}{2})$, it is defined by
	\[
	L_X(t) = \sup\!\left\{
	\frac{\|tx + (1-t)y\|^2 + \|(1-t)x + ty\|^2}{\|x+y\|^2}
	\colon
	\begin{aligned}
		&x,y\in X,\\
		&(x,y)\neq(0,0),\\
		&x\perp_I y
	\end{aligned}
	\right\}.
	\]
	Ni et al. \cite{Ni2025} established the sharp bounds of $L_X(t)$, revealed its intrinsic relationship with the function $\gamma_X(t)$ and the von Neumann-Jordan constant, and used it to characterize the geometric properties such as uniform non-squareness and uniform smoothness of Banach spaces.

	Meanwhile, the equivalence between Rademacher type and Enflo type \cite{Ivanisvili2020} characterizes fundamental metric and linear structural properties of normed spaces. Subsequently, the equivalent form of the constant  with respect to isosceles orthogonality bears certain similarities to the equivalence relation between Rademacher type and Enflo type.

	Motivated by these constants, we define a new symmetric geometric constant $L_p(t,X)$ in complete $p$-normed spaces($0<p\leq 1$), offering an equivalent characterization of the generalized von Neumann-Jordan constant from an orthogonal perspective, as follows:
	\[
	L_p(t,X) = \sup\!\left\{ 
	\frac{\|tx + (1-t)y\|_p^2 + \|(1-t)x + ty\|_p^2}{\|x + y\|_p^2}
	\colon
	\begin{aligned}
		&x, y \in X,\\
		&(x,y) \neq (0,0), \\
		&x \perp_I y\\
		&t \in [0, \tfrac{1}{2})
	\end{aligned}
	\right\}.
	\]
	
	Furthermore, we introduce two Milman-type moduli $J_p(t,X)$ and $S_p(t,X)$ in complete $p$-normed spaces ($0<p\leq 1$), and extend the classical relation between the James constant and the von Neumann-Jordan constant to the parameterized case with corresponding bilateral inequalities. As a natural generalization of Banach spaces, complete $p$-normed spaces have wide applications in analysis, yet their geometric constant theory is underdeveloped, especially the systematic research on isosceles orthogonal constants, which is the gap addressed in this paper. The paper is structured as follows: Section 2 recalls basic concepts including isosceles orthogonality and related geometric constants; Section 3 studies the new symmetric geometric constant, establishes its sharp bounds and equivalent characterizations, reveals its connection with the auxiliary function, and obtains the equivalent characterization of the generalized von Neumann-Jordan constant, with all results reducing to the classical Banach space case when $p=1$; Section 4 investigates the two Milman-type moduli, their properties, sharp bounds, product inequalities, and the extended parameterized relation between the James and von Neumann-Jordan constants.
	
	\section{Preliminaries}
	In this section, we recall some fundamental definitions, notations and basic properties of complete $p$-normed spaces ($0<p\leq 1$), isosceles orthogonality, and classical geometric constants in Banach spaces, which will be used throughout the paper. The theory of complete $p$-normed spaces has been extensively studied and plays an important role in functional analysis (see, e.g., \cite{Bayoumi2003,Huang2017,Albiac2016,Gordon1991,Albiac2010}), while the theory of isosceles orthogonality and related geometric constants can be found in \cite{James1945,Yang2006,Ni2025}.
	
	Throughout this paper, unless otherwise specified, we always denote by $X$ a finite-dimensional real complete $p$-normed space with $0<p\leq1$, whose $p$-norm satisfies the following axioms. We also consistently use $\mathbb{R}$ to denote the set of all real numbers, and assume that the dimension of $X$ satisfies $2\leq \dim X < \infty$ if not otherwise stated.
	
	\begin{definition}\label{def:p-norm}\cite{Bayoumi2003}
		A $p$-norm on a vector space $X$ over $\mathbb{K}$$($scalar field$)$ is a mapping $\|\cdot\|$ from $X$ to $\mathbb{R}_+$ satisfying:
		\begin{enumerate}
			\item[(i)] $\|x\| = 0$ if and only if $x=0$;
			\item[(ii)] $\|\lambda x\| = |\lambda|^p \|x\|$, for every $\lambda \in \mathbb{K}$, $x \in X$;
			\item[(iii)] $\|x + y\| \le \|x\| + \|y\|$, for every $x, y \in X$.
		\end{enumerate}
		A vector space $X$ endowed with a $p$-norm is called a complete $p$-normed space$(0<p\leq 1)$. In this article, the specific $p$-norm on $X$ is denoted by $\norm{\cdot}_p$, to distinguish it from a generic norm. The resulting complete $p$-normed space$(0<p\leq 1) $is denoted by $(X, \norm{\cdot}_p)$.
	\end{definition}

	For a complete $p$-normed space $X$($0<p\leq 1$), we denote the closed unit ball and the unit sphere of $X$ by
	\[
	B_X = \{x\in X: \norm{x}_p \leq 1\}, \quad S_X = \{x\in X: \norm{x}_p = 1\},
	\]

	The study of orthogonality in complete $p$-normed spaces($0<p\leq 1$) represents an interesting and worthwhile research direction.
	
	\begin{definition}\cite{Xiao2025}
		Let $X$ be a complete $p$-normed space$(0<p\leq 1)$. Two elements $x,y\in X$ are said to be isosceles orthogonal, denoted by $x \perp_I y$, if and only if
		\[
		\|x+y\|_p = \|x-y\|_p.\]
	\end{definition}

	Isosceles orthogonality possesses several important properties, in particular, for any $x\in S_X$, there exists $y\in S_X$ such that $x \perp_I y$.

	\begin{definition}\label{def:L-p}
		Let $X$ be a complete $p$-normed space$(0<p\leq 1)$. For any $t \in [0, \frac{1}{2})$, the symmetric geometric constant associated with isosceles orthogonality is defined by
		\begin{equation}\label{eq:def-L}
			L_p(t,X) = \sup\!\left\{
			\frac{\|t x + (1-t)y\|_p^2 + \|(1-t)x + t y\|_p^2}{\|x+y\|_p^2}
			\colon
			\begin{aligned}
				&x, y \in X,\\
				&(x,y) \neq (0,0),\\
				&x \perp_I y
			\end{aligned}
			\right\}.
		\end{equation}
		
		Equivalently, the symmetric geometric constant can be rewritten as follows:
		\begin{equation}\label{eq:def-L-unified}
			\begin{aligned}
				L_p(t,X)
				&= \sup\left\{
				\frac{2\left(\| t x+(1-t) y\|_p^2 + \| (1-t) x+t y\|_p^2\right)}{\| x+y\|_p^2 + \| x-y\|_p^2}
				\colon
				\begin{array}{l}
					x, y \in X,\\
					(x,y) \neq (0,0),\\
					x \perp_I y,\\
					t \in [0, \tfrac{1}{2})
				\end{array}
				\right\}.
			\end{aligned}
		\end{equation}
		where the supremum is taken over all $x, y \in X$ with $(x,y) \neq (0,0)$ and $x \perp_I y$.
	\end{definition}

	Motivated by the function $\gamma_X(t)$ in Banach spaces, we introduce its generalized form in complete $p$-normed spaces($0<p\leq 1$) as follows.
	\begin{definition}\label{def:gamma-p}
		Let $X$ be a complete $p$-normed space$(0<p\leq 1)$. For any $t\in[0,1]$, the function $\gamma_p(t,X)$ is defined by
		\begin{equation}\label{eq:def-gamma-p}
			\gamma_p(t,X) = \supop\left\{ \frac{\norm{x+ty}_p^2 + \norm{x-ty}_p^2}{2} : x,y\in S_X \right\}.
		\end{equation}
	\end{definition}
	
	\begin{remark}
		When $p=1$, the above definition reduces to the function $\gamma_X(t)$ in classical Banach spaces, which is consistent with the limiting case of our subsequent results.
	\end{remark}

	\begin{definition}\label{def:CNJp}
		The  Jordan--von Neumann constant of a $p$-normed space $X$$(0<p\leq 1)$ is defined by
		\begin{equation}\label{eq:def-CNJp}
			C_{\mathrm{NJ},p}(X) = \supop\left\{ \frac{\norm{x+y}_p^2 + \norm{x-y}_p^2}{2\left( \norm{x}_p^2 + \norm{y}_p^2 \right)} : x,y\in X,\ (x,y)\neq(0,0) \right\}.
		\end{equation}
	\end{definition}

	\begin{definition}\label{def:generalized-r-von-Neumann-Jordan}
		Let $X$ be a  $p$-normed space$(0<p\leq 1)$ and $r \geq 2$. The generalized $r$-von Neumann-Jordan constant is defined by
		\begin{equation}\label{eq:def-CNJp-r}
			C_{\mathrm{NJ},p}^{(r)}(X) = \supop\left\{ \frac{\|x+y\|_p^r + \|x-y\|_p^r}{2^{r-1}\left(\|x\|_p^r + \|y\|_p^r\right)} : x,y \in X,\ (x,y)\neq(0,0) \right\}.
		\end{equation}
	\end{definition}

	\section{Basic Properties of $L_p(t,X)$}

In this section, we study the fundamental properties of the symmetric geometric constant $L_p(t,X)$ in complete $p$-normed spaces($0<p\leq 1$). We present its sharp bounds, establish the relationship with the auxiliary function $\gamma_p(t,X)$, and give an equivalent characterization of the generalized von Neumann-Jordan constant $C_{\mathrm{NJ},p}(X)$.

	\begin{theorem}
		Let $X$ be a $p$-normed space$(0<p\leq 1)$. Then for any $t\in[0,\frac{1}{2})$, the following inequalities hold:
		\begin{equation}\label{eq:L-bounds}
			t^{2p} + (1-t)^{2p} \leq L_p(t,X) \leq \frac{1}{2^{2p-1}} \left(1 + (1-2t)^p\right)^{2}.
		\end{equation}
	\end{theorem}
	
	\begin{proof}
		We first prove the lower bound. Take any non-zero vector $x\in X$ and set $y=0$. By the definition of isosceles orthogonality, it is trivial to verify $x\perp_I y$. Substituting into the definition of $L_p(t,X)$, by condition (ii) of Definition \ref{def:p-norm}, we have:
		$$
		\frac{\|tx\|_p^2 + \|(1-t)x\|_p^2}{\|x\|_p^2} = \frac{(t^p\|x\|_p)^2 + ((1-t)^p\|x\|_p)^2}{\|x\|_p^2} = t^{2p} + (1-t)^{2p}.
		$$
		Taking the supremum over all eligible vector pairs yields the lower bound
		$$L_p(t,X) \geq t^{2p} + (1-t)^{2p}.$$
		
		Next, we prove the upper bound. Let $x, y \in X$ satisfy $x \perp_I y$ and $(x,y)\neq(0,0)$. We perform the variable transformation
		$$\alpha = \frac{x+y}{2}, \quad \beta = \frac{x-y}{2}.$$
		By the isosceles orthogonality condition $\|x+y\|_p = \|x-y\|_p$ and the condition (ii) of Definition \ref{def:p-norm}, we have $\|\alpha\|_p = \|\beta\|_p$. Moreover, the linear combinations in the numerator can be rewritten as:
		$$tx+(1-t)y = \alpha - (1-2t)\beta, \quad (1-t)x+ty = \alpha + (1-2t)\beta.$$
		
		Let $s=1-2t$. Since $t\in[0,\frac{1}{2})$, we have $s\in(0,1]$. By the triangle inequality and $p$-homogeneity of the $p$-norm,
		$$
		\|\alpha \pm s\beta\|_p \leq \|\alpha\|_p + \|s\beta\|_p = \|\alpha\|_p + s^p\|\beta\|_p.
		$$
		Combining with $\|\alpha\|_p = \|\beta\|_p$, we get
		$$
		\|\alpha \pm s\beta\|_p \leq \|\alpha\|_p \left(1 + s^p\right) = \|\alpha\|_p \left(1 + (1-2t)^p\right).
		$$
		
		Therefore, the numerator of the ratio satisfies
		$$
		\|\alpha - s\beta\|_p^2 + \|\alpha + s\beta\|_p^2 \leq 2 \cdot \|\alpha\|_p^2 \left(1 + (1-2t)^p\right)^2.
		$$
		Note that $\|x+y\|_p = \|2\alpha\|_p = 2^p \|\alpha\|_p$ by condition (ii) of Definition \ref{def:p-norm}, thus $\|x+y\|_p^2 = 2^{2p} \|\alpha\|_p^2$.
		
		Substituting back into the definition of $L_p(t,X)$, we obtain
		\[
		\begin{aligned}
			\frac{\|tx+(1-t)y\|_p^2 + \|(1-t)x+ty\|_p^2}{\|x+y\|_p^2} 
			&\leq \frac{2 \cdot \|\alpha\|_p^2 \left(1 + (1-2t)^p\right)^2}{2^{2p} \|\alpha\|_p^2} \\
			&= \frac{1}{2^{2p-1}} \left(1 + (1-2t)^p\right)^2.
		\end{aligned}
		\]
		Taking the supremum over all $x,y\in X$ with $(x,y)\neq(0,0)$ and $x\perp_I y$ yields the upper bound. This completes the proof.
	\end{proof}

	In what follows, we further reveal the essential connection between $L_p(t,X)$ and $\gamma_p(t,X)$. Theorems \ref{thm:identity-L-gamma} and \ref{thm:identity-L-CNJ} are both dedicated to this topic, showing that isosceles orthogonality can characterize geometric constants in complete $p$-normed spaces ($0<p\leq 1$), which extends the corresponding conclusions in Banach spaces.

	\begin{theorem}\label{thm:identity-L-gamma}
		Let $X$ be a $p$-normed space $(0<p\leq 1)$. Then the following identity between the symmetric geometric constant $L_p(t,X)$ and the function $\gamma_p(t,X)$ holds:
		\begin{equation}
			L_p(t,X) = \frac{1}{2^{2p-1}} \cdot \gamma_p(1-2t,X).
		\end{equation}
	\end{theorem}
	
	\begin{proof}
		Let $x, y \in X$ such that $x \perp_I y$ and $(x,y) \neq (0,0)$. We choose
		\[
		\alpha = \frac{x+y}{2}, \quad \beta = \frac{x-y}{2}.
		\]
		Then
		\[
		tx+(1-t)y = \alpha - (1-2t)\beta, \quad (1-t)x+ty = \alpha + (1-2t)\beta.
		\]
		By the definition of isosceles orthogonality and condition (ii) in Definition \ref{def:p-norm},
		\[
		\|\alpha\|_p = \frac{1}{2^p}\|x+y\|_p, \quad \|\beta\|_p = \frac{1}{2^p}\|x-y\|_p,
		\]
		thus $\|\alpha\|_p = \|\beta\|_p$, and
		\[
		\frac{\|tx+(1-t)y\|_p^2 + \|(1-t)x+ty\|_p^2}{\|x+y\|_p^2}
		= \frac{\|\alpha - (1-2t)\beta\|_p^2 + \|\alpha + (1-2t)\beta\|_p^2}{2^{2p} \|\alpha\|_p^2}.
		\]
		
		Let $a = \frac{\alpha}{\|\alpha\|_p^{\frac{1}{p}}}$, $b = \frac{\beta}{\|\beta\|_p^{\frac{1}{p}}}$, then $a, b \in S_X$, and
		\[
		\begin{aligned}
			\frac{\|\alpha - (1-2t)\beta\|_p^2 + \|\alpha + (1-2t)\beta\|_p^2}{\|\alpha\|_p^2}
			&= \|a - (1-2t)b\|_p^2 + \|a + (1-2t)b\|_p^2 \\
			&\leq 2 \cdot \gamma_p(1-2t,X).
		\end{aligned}
		\]
		It follows that
		\[
		\frac{\|tx+(1-t)y\|_p^2 + \|(1-t)x+ty\|_p^2}{\|x+y\|_p^2}
		\leq \frac{1}{2^{2p-1}} \gamma_p(1-2t,X).
		\]
		Taking the supremum over all eligible $x,y$, we obtain
		\[
		L_p(t,X) \leq \frac{1}{2^{2p-1}} \gamma_p(1-2t,X).
		\]
		
		On the other hand, let $a, b \in S_X$. We choose $\alpha = \frac{a+b}{2}$, $\beta = \frac{a-b}{2}$, then $\alpha+\beta = a \in S_X$ and $\alpha-\beta = b \in S_X$. Since
		\[
		\begin{aligned}
			&\frac{\|a - (1-2t)b\|_p^2 + \|a + (1-2t)b\|_p^2}{2} \\
			=&\ \frac{\|(\alpha+\beta) - (1-2t)(\alpha-\beta)\|_p^2 + \|(\alpha+\beta) + (1-2t)(\alpha-\beta)\|_p^2}{2} \\
			=&\ 2^{2p-1} \cdot \frac{\|t\alpha + (1-t)\beta\|_p^2 + \|(1-t)\alpha + t\beta\|_p^2}{\|\alpha+\beta\|_p^2} \\
			\leq&\ 2^{2p-1} L_p(t,X).
		\end{aligned}
		\]
		that is
		\[
		\frac{\|a - (1-2t)b\|_p^2 + \|a + (1-2t)b\|_p^2}{2} \leq 2^{2p-1} L_p(t,X).
		\]
		Taking the supremum over all $a,b \in S_X$ yields
		\[
		\gamma_p(1-2t,X) \leq 2^{2p-1} L_p(t,X),
		\]
		which is equivalent to
		\[
		L_p(t,X) \geq \frac{1}{2^{2p-1}} \gamma_p(1-2t,X).
		\]
		Combining both inequalities completes the proof.
	\end{proof}

	In complete $p$-normed spaces ($0<p\leq 1$), we can also establish an equivalent characterization between the isosceles orthogonal constant and the generalized von Neumann-Jordan constant, which shows that isosceles orthogonality is also applicable to characterizing the equivalent forms of such geometric constants.
	\begin{theorem}\label{thm:identity-L-CNJ}
	Let $X$ be a complete $p$-normed space $(0<p\leq 1)$. Then the following result holds for its generalized von Neumann-Jordan constant:
		\begin{equation}
			C_{\mathrm{NJ},p}(X) = \sup_{\eta \in [0,1]} \frac{2^{2p-1} \cdot L_p\left( \frac{1-\eta}{2} ,X\right)}{1 + \eta^{2p}}.
		\end{equation}
		
	\end{theorem}
	
	\begin{proof}
		First, recall the equivalent definition of the generalized von Neumann-Jordan constant:
		\[
		C_{\mathrm{NJ},p}(X) = \sup_{\eta \in [0,1]} \frac{\gamma_p(\eta,X)}{1+\eta^{2p}}.
		\]
		Combined with Theorem \ref{thm:identity-L-gamma}, we have
		\[
		L_p(t,X) = \frac{1}{2^{2p-1}} \cdot \gamma_p(1-2t,X).
		\]
		Let $1-2t = \eta$, then $\eta \in [0,1]$ and $t = \frac{1-\eta}{2}$, which gives
		\[
		\gamma_p(\eta,X) = 2^{2p-1} \cdot L_p\left( \frac{1-\eta}{2},X \right).
		\]
		Substituting this into the expression of $C_{\mathrm{NJ},p}(X)$, we obtain
		\[
		C_{\mathrm{NJ},p}(X) = \sup_{\eta \in [0,1]} \frac{2^{2p-1} \cdot L_p\left( \frac{1-\eta}{2},X \right)}{1 + \eta^{2p}}.
		\]
		
		which completes the proof.
	\end{proof}

	\section{Milman-Type Moduli and Parameterized Geometric Constants}

	Motivated by the James constant \(J_p(X)\) and Schäffer constant \(S_p(X)\) defined above, we introduce the following parametric Milman-type moduli in \(p\)-normed spaces, which extend \(J_p(X)\) and \(S_p(X)\) to the case involving a real parameter \(t\geq0\).
	\begin{definition}\label{def:Milman-moduli}
		Let $(X,\|\cdot\|_{p})$ be a complete $p$-normed space$(0<p\leq 1)$. For $t \geq0$, the Milman-type moduli are defined by
		\begin{equation}\label{eq:def-Jp-t}
			J_{p}(t,X)=\sup\left\{ \min\{ \| x+ty\| _{p},\| x-ty\| _{p}\} : x,y\in S_{X}\right\},
		\end{equation}
		and
		\begin{equation}\label{eq:def-Sp-t}
			S_{p}(t,X)=\inf\left\{ \max\{ \| x+ty\| _{p},\| x-ty\| _{p}\} : x,y\in S_{X}\right\}.
		\end{equation}
	\end{definition}
	
	\begin{remark}
		\begin{enumerate}
			\item[(i)] If Y is a subspace of X , then it is clear that \(J_{p}(t, X) \geq J_{p}(t, Y)\) and \(S_{p}(t, X) \leq S_{p}(t, Y)\).
			\item[(ii)] If \(\dim X=1\) , then for any \(t \geq0\) and \(0<p \leq1\) , \(J_{p}(t, X)=|1-t|^{p}\) and \(S_{p}(t, X)=(1+t)^{p}\).
		\end{enumerate}
	\end{remark}

	\begin{proposition}

		Let \((X,\|\cdot\|_{p})\) be a complete $p$-normed  space  . Then
		\[
		J_{p}(t, X)=t^{p} J_{p}\left(\frac{1}{t}, X\right) \quad \text{and} \quad S_{p}(t, X)=t^{p} S_{p}\left(\frac{1}{t}, X\right).
		\]
	\end{proposition}
	
	\begin{proof}
		Note that, for any \(x, y \in S_{X}\) , we have
		\[			
		\min \{ \| x+ty\| _{p},\| x-ty\| _{p}\} =t^{p}\min\left\{ \left\| \frac {1}{t}x+y\right\| _{p},\left\| \frac {1}{t}x-y\right\| _{p}\right\}.
		\]	
		Then by the definition of \(J_{p}(t, X)\) and the arbitrariness of x and y , we get that \(J_{p}(t, X)=t^{p} J_{p}(\frac{1}{t}, X)\) . Similarly, we can obtain that \(S_{p}(t, X)=t^{p} S_{p}(\frac{1}{t}, X)\) .
	\end{proof}
	
	\begin{proposition}\label{222}
		Let \((X,\|\cdot\|_{p})\) be a complete  $p$-normed  space \((0<p \leq1)\)  . Then
		\begin{enumerate}
			\item[(i)] \(\min\left\{1, t^{p}\right\} 2^{p-1} \leq S_{p}(t, X) \leq(1+t)^{p}\).
			\item[(ii)] \(|1-t|^{p} \leq J_{p}(t, X) \leq 1+t^{p}\).
		\end{enumerate}
	\end{proposition}
	
	\begin{proof}
		(i) First, since
		\[				
		\begin{aligned}
			\max \left\{\| x+t y\| _{p},\| x-t y\| _{p}\right\} & \geq \frac{1}{2}\left(\| x+t y\| _{p}+\| x-t y\| _{p}\right) \\
			& \geq \frac{1}{2}\left(\| (x+t y)+|x-t y|\| _{p}\right) \\
			& \geq \min \left\{1, t^{p}\right\} 2^{p-1} .
		\end{aligned}
		\]	
		This implies that \(S_{p}(t, X) \geq \min \{1, t^{p}\} 2^{p-1}\).
		
		Next, take \(y=x \in S_{X}\) , we get that \(\max {\|x+t y\|_{p},\|x-t y\|_{p}}=(1+t)^{p}\) , which means that \(S_{p}(t, X) \leq(1+t)^{p}\).
		
		(ii) On the one hand, we also let \(y=x\) , then we can easily get that
		\[					
		\min \left\{\| x+t y\| _{p},\| x-t y\| _{p}\right\}=|1-t|^{p} .
		\]				
		This implies that \(J_{p}(t, X) \geq|1-t|^{p}\) .
		
		On the other hand, since \(\|x+t y\|_{p},\|x-t y\|_{p} \leq\|x\|_{p}+\|t y\|_{p}=1+t^{p}\) , it follows that \(J_{p}(t, X) \leq 1+t^{p}\) as desired.
	\end{proof}
	
	The following example will show that the upper bound of \(J_{p}(t, X)\) and \(S_{p}(t, X)\) can be reached.
	
	\begin{example}
		Let $0<p\leq1$. In \cite{Xiao2025}, the authors have shown that $l^p$ is a complete $p$-normed space$(0<p\leq 1)$. For this space, we have $J_p(t, l^p) = 1+t^p$ and $S_p(t, l^p) = (1+t)^p$ for all $t\geq0$, and in particular $J_p(1, l^p)=S_p(1, l^p)=2$.
	\end{example}
	
	\begin{proof}
		First, we prove $J_p(t, l^p) = 1 + t^p$.
		Let $x=(1,0,0,\cdots)$ and $y=(0,1,0,\cdots)$, then $x, y \in S_{l^p}$.
		By the definition of the $p$-norm, we have
		$$
		\|x+ty\|_p = |1|^p + |t|^p = 1 + t^p, \quad \|x-ty\|_p = |1|^p + |-t|^p = 1 + t^p.
		$$
		Thus
		$$
		\min\left\{ \|x+ty\|_p, \|x-ty\|_p \right\} = 1 + t^p,
		$$
		which implies $J_p(t,l^p) \geq 1 + t^p$.
		Combining with the upper bound $J_p(t,X) \leq 1 + t^p$ from \ref{222} (ii), we get $J_p(t,l^p) = 1 + t^p$.
		
		Next, we prove $S_p(t, l^p) = (1+t)^p$.
		Take $y=x \in S_{l^p}$, then
		$$
		\|x+ty\|_p = \|(1+t)x\|_p = (1+t)^p \|x\|_p = (1+t)^p, \quad \|x-ty\|_p = \|(1-t)x\|_p = |1-t|^p.
		$$
		Thus
		$$
		\max\left\{ \|x+ty\|_p, \|x-ty\|_p \right\} = (1+t)^p,
		$$
		which implies $S_p(t,l^p) \geq (1+t)^p$.
		Combining with the upper bound $S_p(t,X) \leq (1+t)^p$ from Proposition 4.4 (i), we get $S_p(t,l^p) = (1+t)^p$.
	\end{proof}
	
	To further explore the equivalent characterizations of the two moduli \(J_p(t,X)\) and \(S_p(t,X)\), we present the following two lemmas, which extend the domain of the vectors in the definition from the unit sphere to the unit ball and the whole space, and lay a foundation for the proof of the subsequent core theorems.
	
	\begin{lemma}
		Let $(X,\|\cdot\|_{p})$ be a complete $p$-normed space$(0<p\leq 1)$. Then
		\[				
		\begin{aligned}
			J_{p}(t, X) & =\sup \left\{\min \left\{\| x+t y\| _{p},\| x-t y\| _{p}\right\}: x, y \in X,\| x\| _{p}=\| y\| _{p}=a^{p} \leq 1\right\} \\
			& =\sup \left\{\min \left\{\| x+t y\| _{p},\| x-t y\| _{p}\right\}: x, y \in B_{X}\right\} .
		\end{aligned}
		\]	
		Here, $a \in [0,1]$ is a non-negative real number, and $a^p$ denotes the common value of the $p$-norms of $x$ and $y$.
	\end{lemma}
	
	\begin{proof}
		For the convenience of proof, in the following discussion, we denote
		\[
		J_{p}^{a}(t,X)=\sup\left\{ \min\{ \| x+ty\| _{p},\| x-ty\| _{p}\} :x,y\in X,\| x\| _{p}=\| y\| _{p}=a^{p}\leq 1\right\}
		\]
		and
		\[
		J_{p}'(t,X)=\sup\left\{ \min\left\{ \| x+ty\| _{p},\| x-ty\| _{p}\right\} : x,y\in B_{X}\right\} .
		\]
		
		First, by the definitions of the unit sphere $S_X$ and unit ball $B_X$, for any fixed $a\in[0,1]$, we have
		\[
		\begin{aligned}
			\bigl\{ (x,y) : x,y\in S_X \bigr\}
			&\subset \bigl\{ (x,y) : x,y\in X,\ \|x\|_{p}=\|y\|_{p}=a^{p}\leq 1 \bigr\} \\
			&\subset \bigl\{ (x,y) : x,y\in B_X \bigr\}.
		\end{aligned}
		\]
		Thus it follows that
		\[
		\begin{aligned}
			& \left\{ \min\left\{ \|x+ty\|_p, \|x-ty\|_p \right\} : x,y\in B_X \right\} \\
			\supset & \left\{ \min\left\{ \|x+ty\|_p, \|x-ty\|_p \right\} : x,y\in X,\ \|x\|_p = \|y\|_p = a^p \le 1 \right\} \\
			\supset & \left\{ \min\left\{ \|x+ty\|_p, \|x-ty\|_p \right\} : x,y\in S_X \right\}.
		\end{aligned}
		\]

		Furthermore, we have
		\[
		\begin{aligned}
			&J_{p}(t, X) \\
			=& \sup \left\{\min \left\{\left\| \frac{x}{a}+\frac{t y}{a}\right\| _{p},\left\| \frac{x}{a}-\frac{t y}{a}\right\| _{p}\right\}: x, y \in X,\| x\| _{p}=\| y\| _{p}=a^{p} \leq 1\right\} \\
			=& \frac{1}{a^{p}} \sup \left\{\min \left\{\| x+t y\| _{p},\| x-t y\| _{p}\right\}: x, y \in X,\| x\| _{p}=\| y\| _{p}=a^{p} \leq 1\right\} \\
			\leq& \sup \left\{\min \left\{\| x+t y\| _{p},\| x-t y\| _{p}\right\}: x, y \in X,\| x\| _{p}=\| y\| _{p}=a^{p} \leq 1\right\} \\
			=& J_{p}^{a}(t, X)
		\end{aligned}
		\]
		which means that \(J_{p}(t, X)=J_{p}^{a}(t, X)\) and this indicates that \(J_{p}^{a}(X)\) is unrelated to a . Thus, by the definition of \(J_{p}'(t, X)\) , let \(a_{n}=1+\frac{1}{n}\) , then for any \(n \in \mathbb{N}^{*}\) , there exist \(x_{n}, y_{n} \in X\) such that
		\[
		1 \leq\left\| x_{n}\right\| _{p} \leq a_{n}^{p}, \quad 1 \leq\left\| y_{n}\right\| _{p} \leq a_{n}^{p},
		\]
		and
		\[				
		J_{p}'(t, X)-\frac{1}{n} \leq \min \left\{\left\| x_{n}+t y_{n}\right\| _{p},\left\| x_{n}-t y_{n}\right\| _{p}\right\}<J_{p}'(t, X) .
		\]
		Thus, we obtain two sequences \({x_{n}}\) and \({y_{n}}\) in \({x \in X: 1 \leq\|x\|_{p} \leq2^{p}}\) Since \(\dim X<\infty\) , \({x \in X: 1 \leq\|x\|_{p} \leq2^{p}}\) is compact. This means that \({x_{n}}\) and \({y_{n}}\) have convergent subsequences. For the sake of brevity, we suppose that \(\displaystyle\lim _{n \to \infty} x_{n}=\tilde{x}\) and \(\displaystyle\lim _{n \to \infty} y_{n}=\tilde{y}\) . Hence, from (1), it follows that \(\tilde{x}, \tilde{y} \in S_{X}\) and
		\[
		J_{p}'(t, X)=\min \left\{\| \tilde{x}+t \tilde{y}\| _{p},\| \tilde{x}-t \tilde{y}\| _{p}\right\} \leq J_{p}(t, X),
		\]
		
		Hence, we obtain that \(J_{p}(t, X)=J_{p}^{a}(t, X)=J_{p}'(t, X)\).
	\end{proof}
	
	\begin{lemma}
		Let \((X,\|\cdot\|_{p})\) be a complete $p$-normed space$(0<p\leq 1)$. Then
		\[			
		\begin{aligned}
			S_{p}(t, X) & =\inf \left\{\max \left\{\| x+t y\| _{p},\| x-t y\| _{p}\right\}: x, y \in X,\| x\| _{p}=\| y\| _{p}=b^{p} \geq 1\right\} \\
			& =\inf \left\{\max \left\{\| x+t y\| _{p},\| x-t y\| _{p}\right\}: x, y \in X,\| x\| _{p},\| y\| _{p} \geq 1\right\} .
		\end{aligned}
		\]
	\end{lemma}
	
	\begin{proof}
		For the convenience of proof, in the following discussion, we denote
		\[
		S_{p}^{b}(t,X)=\inf\left\{ \max\{ \| x+ty\| _{p},\| x-ty\| _{p}\} : x,y\in X,\| x\| _{p}=\| y\| _{p}=b^{p}\geq 1\right\}
		\]
		and
		\[
		S_{p}'(t,X)=\inf\bigg \{ \max\{ \| x+ty\| _{p},\| x-ty\| _{p}\} : x,y\in X,\| x\| _{p},\| y\| _{p}\geq 1\bigg \} .
		\]
		
		First, since
		\[
		\{ x,y\in S_{X}\} \subseteq \{ x,y\in X,\| x\| _{p}=\| y\| _{p}=b^{p}\geq 1\} \subseteq \{ x,y\in X,\| x\| _{p},\| y\| _{p}\geq 1\} ,
		\]
		we get that \(S_{p}(t, X) \geq S_{p}^{b}(t, X) \geq S_{p}'(t, X)\).
		
		Furthermore, we have
		\[
		\begin{aligned}
			S_{p}(t, X) & =\inf \left\{\max \left\{\left\| \frac{x}{b}+\frac{t y}{b}\right\| _{p},\left\| \frac{x}{b}-\frac{t y}{b}\right\| _{p}\right\}: x, y \in X,\| x\| _{p}=\| y\| _{p}=b^{p} \geq 1\right\} \\
			& =\frac{1}{b^{p}} \inf \left\{\max \left\{\| x+t y\| _{p},\| x-t y\| _{p}\right\}: x, y \in X,\| x\| _{p}=\| y\| _{p}=b^{p} \geq 1\right\} \\
			& \leq \inf \left\{\max \left\{\| x+t y\| _{p},\| x-t y\| _{p}\right\}: x, y \in X,\| x\| _{p}=\| y\| _{p}=b^{p} \geq 1\right\} \\
			& =S_{p}^{b}(t, X),
		\end{aligned}
		\]
		which means that \(S_{p}(t, X)=S_{p}^{b}(t, X)\) , and this indicates that \(S_{p}^{b}(X)\) is unrelated to b. Thus, by the definition of \(S_{p}'(t, X)\) , let \(b_{n}=1+\frac{1}{n}\) , then for any \(n \in \mathbb{N}^{*}\) , there exist \(x_{n}, y_{n} \in X\) such that
		\[
		1 \leq\left\| x_{n}\right\| _{p} \leq b_{n}^{p},\quad 1 \leq\left\| y_{n}\right\| _{p} \leq b_{n}^{p},
		\]
		and
		\begin{equation}
			S_{p}'(t, X) \leq \max \left\{\left\| x_{n}+t y_{n}\right\| _{p},\left\| x_{n}-t y_{n}\right\| _{p}\right\}<S_{p}'(t, X)+\frac{1}{n} .
		\end{equation}
		
		Thus, we obtain two sequences \({x_{n}}\) and \({y_{n}}\) in \({x \in X: 1 \leq\|x\|_{p} \leq2^{p}}\) Since \(\dim X<\infty\) , \({x \in X: 1 \leq\|x\|_{p} \leq2^{p}}\) is compact. This means that \({x_{n}}\) and \({y_{n}}\) have convergent subsequences. For the sake of brevity, we suppose that \(\displaystyle\lim _{n \to \infty} x_{n}=\tilde{x}\) and \(\displaystyle\lim _{n \to \infty} y_{n}=\tilde{y}\) . Hence, from (3), it follows that \(\tilde{x}, \tilde{y} \in S_{X}\) and
		\[
		S_{p}'(t,X)=\max\{ \| \tilde {x}+t\tilde {y}\| _{p},\| \tilde {x}-t\tilde {y}\| _{p}\} \geq S_{p}(t,X),
		\]
		
		Hence, we obtain that \(S_{p}(t, X)=S_{p}^{b}(t, X)=S_{p}'(t, X)\).
	\end{proof}
	
	With the equivalent characterizations established in the above lemmas, we can now prove the core inequalities for the product of \(J_p(t,X)\) and \(S_p(t,X)\), which reveal the intrinsic relationship between these two moduli.
	
	\begin{theorem}
		Let \((X,\|\cdot\|_{p})\) be a complete p-normed space. Then
		\begin{equation}
			\min\left\{ 1,t^{p}\right\} 2^{p}\leq J_{p}(t,X)S_{p}(t,X)\leq \max\{ 1,t^{p}\} 2^{p}.
		\end{equation}
	\end{theorem}
	
	\begin{proof}
		By the definition of the \(J_{p}(t, X)\) constant, let \(a_{n}=1+\frac{1}{n}\) , then for any \(n \in \mathbb{N}^{*}\) , there exist \(x_{n}\) , \(y_{n} \in S_{X}\) such that
		\begin{equation}\label{eq:thm48-1}
			J_{p}(t, X)-\frac{1}{n} \leq \min \left\{\left\| x_{n}+t y_{n}\right\| _{p},\left\| x_{n}-t y_{n}\right\| _{p}\right\}=a_{n}^{p} \leq J_{p}(t, X) .
		\end{equation}
		
		From this, we know that \(\left\|\frac{x_{n}+t y_{n}}{a_{n}}\right\|_{p} \geq1\) and \(\left\|\frac{x_{n}-t y_{n}}{a_{n}}\right\|_{p} \geq1\) , thus, by Lemma 4.7, we have
		\begin{equation}\label{eq:thm48-2}
			\begin{array} {rl}{S_{p}(t,X)}&{=\inf\{ \max\left\{ \| x+ty\| _{p},\| x-ty\| _{p}\} :x,y\in X,\| x\| _{p},\| y\| _{p}\geq 1\right\} }\\ &{\leq \max\left\{ \left\| \frac {x_{n}+ty_{n}}{a_{n}}+\frac {x_{n}-ty_{n}}{a_{n}}\right\| _{p},\left\| \frac {x_{n}+ty_{n}}{a_{n}}-\frac {x_{n}-ty_{n}}{a_{n}}\right\| _{p}\right\} }\\ &{=\frac {\max\{ 1,t^{p}\} 2^{p}}{a_{n}^{p}}.}\end{array}
		\end{equation}
		
		Combine \eqref{eq:thm48-1} and \eqref{eq:thm48-2}, we obtain that \(J_{p}(t, X) S_{p}(t, X) \leq \max \{1, t^{p}\} 2^{p}\).
		
		Furthermore, by the definition of the \(S_{p}(t, X)\) constant, for any \(n \in \mathbb{N}^{*}\) , there exist \(x_{n}, y_{n} \in S_{X}\) such that
		\begin{equation}\label{eq:thm48-3}
			S_{p}(t,X)\leq \max\left\{ \| x_{n}+ty_{n}\| _{p},\| x_{n}-ty_{n}\| _{p}\right\} =b_{n}^{p}<S_{p}(t,X)+\frac {1}{n}.
		\end{equation}
		
		From this, we know that \(\left\|\frac{x_{n}+t y_{n}}{b_{n}}\right\|_{p} \leq1\) and \(\left\|\frac{x_{n}-t y_{n}}{b_{n}}\right\|_{p} \leq1\) , thus, by Lemma 4.6, we have
		\begin{equation}\label{eq:thm48-4}
			\begin{array} {rl}
				J_{p}(t,X)&=\sup\left\{ \min\left\{ \| x+ty\| _{p},\| x-ty\| _{p}\right\} :x,y\in B_{X}\right\} \\
				&\geq \min\left\{ \left\| \frac {x_{n}+ty_{n}}{b_{n}}+\frac {x_{n}-ty_{n}}{b_{n}}\right\| _{p},\left\| \frac {x_{n}+t y_{n}}{b_{n}}-\frac {x_{n}-ty_{n}}{b_{n}}\right\| _{p}\right\} \\
				&=\frac {\min\{ 1,t^{p}\} 2^{p}}{b_{n}^{p}}.
			\end{array}
		\end{equation}
		
		Combine \eqref{eq:thm48-3} and \eqref{eq:thm48-4}, it follows that \(J_{p}(t, X) S_{p}(t, X) \geq \min \{1, t^{p}\} 2^{p}\) , as desired.
	\end{proof}
	
	\begin{theorem}
		Let \((X,\|\cdot\|_{p})\) be a complete  $p$-normed  space .Then
		\begin{enumerate}
			\item[(i)] \(\min\{ t^{p},1\} J_{p}(X)-|1-t|^{p}\leq J_{p}(t,X)\leq \max\{ 1,t^{p}\} J_{p}(X)+|1-t|^{p}.\)
			\item[(ii)] \(\min\{ t^{p},1\} S_{p}(X)-|1-t|^{p}\leq S_{p}(t,X)\leq \max\{ 1,t^{p}\} S_{p}(X)+|1-t|^{p}.\)
		\end{enumerate}
	\end{theorem}
	
	\begin{proof}
		(i)First, for any \(x, y \in S_{X}\) and \(t>0\) , we have
		\[
		\begin{aligned}
			t^{p} \min \left\{\| x+y\| _{p},\| x-y\| _{p}\right\} & = \min \left\{\| t x+t y\| _{p},\| t x-t y\| _{p}\right\} \\
			& \leq \min \left\{|1-t|^{p}+\| x+t y\| _{p},|1-t|^{p}+\| x-t y\| _{p}\right\} \\
			& =|1-t|^{p}+ \min \left\{\| x+t y\| _{p},\| x-t y\| _{p}\right\} \\
			& \leq|1-t|^{p}+J_{p}(t, X),
		\end{aligned}
		\]
		which means that
		\begin{equation}\label{eq:thm49-1}
			t^{p} J_{p}(X)-|1-t|^{p} \leq J_{p}(t, X) .
		\end{equation}
		
		Furthermore, for any \(x, y \in S_{X}\) and \(t>0\) , we have
		\[
		\begin{array} {rl}{\min\{ \| x+y\| _{p},\| x-y\| _{p}\} }&{\leq \min\{ |1-t|^{p}+\| x+ty\| _{p},|1-t|^{p}+\| x-ty\| _{p}\} }\\ &{=|1-t|^{p}+\min\{ \| x+ty\| _{p},\| x-ty\| _{p}\} }\\ &{\leq |1-t|^{p}+J_{p}(t,X),}\end{array}
		\]
		which means that
		\begin{equation}\label{eq:thm49-3}
			J_{p}(X)-|1-t|^{p} \leq J_{p}(t, X) .
		\end{equation}
		
		Combine the above \eqref{eq:thm49-1} and \eqref{eq:thm49-3}, we get that
		\[
		J_{p}(t,X)\geq \min\{ t^{p},1\} J_{p}(X)-|1-t|^{p}.
		\]
		
		On the other hand, for any \(x, y \in S_{X}\) and \(t>0\) , we have
		\begin{align*}
			&\min\{ \| x+ty\| _{p},\| x-ty\| _{p}\} \\
			=& \min\{ \| x+y+(t-1)y\| _{p},\| x-y+(1-t)y\| _{p}\} \\
			\leq & \min\{ \| x+y\| _{p},\| x-y\| _{p}\} +|1-t|^{p}
		\end{align*}
		and
		\begin{align*}
			&\min\{ \| x+ty\| _{p},\| x-ty\| _{p}\} \\
			=& \min\{ \| tx+ty+(1-t)x\| _{p},\| tx-ty+(1-t)x\| _{p}\} \\
			\leq & t^{p}\min\{ \| x+y\| _{p},\| x-y\| _{p}\} +|1-t|^{p}.
		\end{align*}
		
		This implies that \(J_{p}(t, X) \leq \max \{1, t^{p}\} J_{p}(X)+|1-t|^{p}\).
		
		(ii)Similarly, we can easily draw this conclusion in the same way as (i).
	\end{proof}

The following theorem adopts the proof method from \cite{Kato2001} and extends the relevant results to complete $p$-normed spaces ($0 < p \leq 1$). The proof retains the original framework and makes slight modifications adapted to the definition of $p$-normed spaces($0 < p \leq 1$).
\begin{theorem}
	Let $(X,\norm{\cdot}_p)$ be a complete $p$-normed space. Then
	\begin{equation}\label{eq:J-CNJ-inequality}
		\frac{1}{2}J_p(X)^2\leq C_{\mathrm{NJ},p}(X) \leq \frac{J_p(X)^2}{(J_p(X)-1)^2+1}.
	\end{equation}
\end{theorem}

\begin{proof}
	We first prove the left inequality. For any $x,y\in S_X$, we have
	\begin{align*}
		\minop\left\{ \norm{x+y}_p^2, \norm{x-y}_p^2 \right\}
		&\leq \frac{\norm{x+y}_p^2 + \norm{x-y}_p^2}{2} \\
		&= 2 \frac{\norm{x+y}_p^2 + \norm{x-y}_p^2}{2\left( \norm{x}_p^2 + \norm{y}_p^2 \right)} \\
		&\leq 2 C_{\mathrm{NJ},p}(X),
	\end{align*}

	which implies the left inequality of \eqref{eq:J-CNJ-inequality} . To prove the right inequality we
	first observe that
	
	\[
	C_{\mathrm{NJ},p}(X) = \supop\left\{ \frac{\norm{x+y}_p^2 + \norm{x-y}_p^2}{2\left( \norm{x}_p^2 + \norm{y}_p^2 \right)} : x,y\in X,\ \norm{x}_p=1,\ \norm{y}_p\leq1 \right\}.
	\]
	In fact, if $0\neq \norm{x}_p \geq \norm{y}_p$ (for $\norm{x}_p \leq \norm{y}_p \neq 0$ the proof is similar by symmetry), then
	\[
	\norm{x\pm y}_p = \left\| \norm{x}_p^{\frac{1}{p}} \cdot \frac{x}{\norm{x}_p^{\frac{1}{p}}} \pm \norm{x}_p^{\frac{1}{p}} \cdot \frac{y}{\norm{x}_p^{\frac{1}{p}}} \right\|_p = \norm{x}_p \left\| \frac{x}{\norm{x}_p^{\frac{1}{p}}} \pm \frac{y}{\norm{x}_p^{\frac{1}{p}}} \right\|_p,
	\]
	and
	\[
	\norm{y}_p = \left\| \norm{x}_p^{\frac{1}{p}} \cdot \frac{y}{\norm{x}_p^{\frac{1}{p}}} \right\|_p = \norm{x}_p \left\| \frac{y}{\norm{x}_p^{\frac{1}{p}}} \right\|_p,
	\]
	so $\left\| \frac{y}{\norm{x}_p^{\frac{1}{p}}} \right\|_p = \frac{\norm{y}_p}{\norm{x}_p} \leq 1$. Substituting into the ratio gives
	\begin{align*}
		\frac{\norm{x+y}_p^2 + \norm{x-y}_p^2}{2\left( \norm{x}_p^2 + \norm{y}_p^2 \right)}
		&= \frac{\norm{x}_p^2 \left( \left\| \frac{x}{\norm{x}_p^{\frac{1}{p}}} + \frac{y}{\norm{x}_p^{\frac{1}{p}}} \right\|_p^2 + \left\| \frac{x}{\norm{x}_p^{\frac{1}{p}}} - \frac{y}{\norm{x}_p^{\frac{1}{p}}} \right\|_p^2 \right)}{2 \norm{x}_p^2 \left( 1 + \left\| \frac{y}{\norm{x}_p^{\frac{1}{p}}} \right\|_p^2 \right)} \\
		&= \frac{\left\| \frac{x}{\norm{x}_p^{\frac{1}{p}}} + \frac{y}{\norm{x}_p^{\frac{1}{p}}} \right\|_p^2 + \left\| \frac{x}{\norm{x}_p^{\frac{1}{p}}} - \frac{y}{\norm{x}_p^{\frac{1}{p}}} \right\|_p^2}{2\left( 1 + \left\| \frac{y}{\norm{x}_p^{\frac{1}{p}}} \right\|_p^2 \right)},
	\end{align*}
	which shows that the supremum in the definition of $C_{\mathrm{NJ},p}(X)$ can be taken just over $x,y \in X$ such that $\norm{x}_p = 1$ and $\norm{y}_p \leq 1$.
	
	To obtain the right estimate, we consider two cases (of course for $\norm{x}_p = 1$ and $\norm{y}_p \leq 1$):
	
	\textbf{Case 1:} $\norm{y}_p = q \geq J_p(X) - 1$. Then
	\begin{align*}
		A &:= \frac{\norm{x+y}_p^2+\norm{x-y}_p^2}{2\left(\norm{x}_p^2+\norm{y}_p^2\right)} \\
		&\leq \frac{\left(\norm{x}_p+\norm{y}_p\right)^2 + \left(\min\left\{ \norm{x+y}_p, \norm{x-y}_p \right\}\right)^2}{2\left(\norm{x}_p^2+\norm{y}_p^2\right)} \\
		&\leq \frac{(1+q)^2 + J_p(X)^2}{2(1+q^2)} =: f(q).
	\end{align*}
	Since
	\[
	f'(q)=\frac{1 - qJ_p(X)^2 - q^2}{(1+q^2)^2},
	\]
	it follows that $f$ is increasing on $(0,q_0)$ and decreasing on $(q_0,1)$, where
	\[
	q_0=\frac{-J_p(X)^2+\sqrt{J_p(X)^4+4}}{2}.
	\]
	Now $J_p(X)-1 \geq q_0$ and so
	\[
	A \leq f(q) \leq f(J_p(X)-1) = \frac{J_p(X)^2}{(J_p(X)-1)^2+1}.
	\]
	
	\textbf{Case 2:} $\norm{y}_p = q \leq J_p(X) - 1$. Then
	\[
	A \leq \frac{2\left(\norm{x}_p+\norm{y}_p\right)^2}{2\left(\norm{x}_p^2+\norm{y}_p^2\right)} = \frac{(1+q)^2}{1+q^2} =: g(q).
	\]
	Since $g'(q)=\frac{2(1-q^2)}{(1+q^2)^2}>0$, it follows that $g$ is increasing on $(0,1]$ and so
	\[
	A \leq g(q) \leq g(J_p(X)-1) = \frac{J_p(X)^2}{(J_p(X)-1)^2+1}.
	\]
	
	Thus, in both cases,
	\[
	A \leq \frac{J_p(X)^2}{(J_p(X)-1)^2+1}.
	\]
	This completes the proof.
\end{proof}

	To prepare for the proof of the main theorem, we first present a key auxiliary lemma. The subsequent theorem follows the research method of \cite{Yang2017}.
	
\begin{lemma}\label{111}
	Let $0 < p \leq 1$ and $r \geq 2$, and let $J=J_p(X)\in[J_2,2]$ satisfy both inequalities $\left(\frac{J}{2}\right)^r + 1 \leq J$ and $1+\frac{2+J(J^r+2^r-1)}{2+J}\leq J$.
	Here $J_1$ is the minimal solution of the first inequality, $J_2$ is the minimal solution of the second inequality with $1<J_2<2$, and $J_1\leq J_2$.
	
	Then for any $t \in [0,1]$, we have
	\[
	\phi(t) = \frac{\left[J t^p + (1-t)^p\right]^r + \left[2 t^p + (1-t)^p\right]^r}{2^{r-1}\left(1 + t^{pr}\right)} \leq J.
	\]
\end{lemma}
	
	\begin{proof}
		We first verify the inequality at the endpoints of $[0,1]$.
		
		When $t = 0$,
		\[
		\phi(0) = \frac{1^r + 1^r}{2^{r-1}} = 2^{2-r} \leq \frac{1}{2} < J,
		\]
		since $J \geq J_2 > 1 > \frac{1}{2}$.
		
		When $t = 1$,
		\[
		\phi(1) = \frac{J^r + 2^r}{2^r} = \left(\frac{J}{2}\right)^r + 1.
		\]
		Since $J \in [J_2, 2]$ and $J_2 \geq J_1$, we have $\left(\frac{J}{2}\right)^r + 1 \leq J$, which implies $\phi(1) \leq J$.
		
		Next, let $t_0 \in (0,1)$ be a stationary point of $\phi(t)$. We consider the following two cases:
		
		\medskip
		\noindent \textbf{Case 1:} $t_0^{pr} \leq \frac{J}{2}$
		\begin{align*}
			\phi(t_0) &\leq \frac{\left[1 + J t_0^p\right]^r + \left[1 + 2 t_0^p\right]^r}{2^{r-1}(1 + t_0^{pr})} \\
			&\leq \frac{2^{r-1}(1 + J^r t_0^{pr}) + 2^{r-1}(1 + 2^r t_0^{pr})}{2^{r-1}(1 + t_0^{pr})} \\
			&\leq 1 + \frac{2 + J(J^r + 2^r - 1)}{2 + J}.
		\end{align*}
		By the definition of $J_2$, for all $J \in [J_2, 2]$, we have
		\[
		1 + \frac{2 + J(J^r + 2^r - 1)}{2 + J} \leq J,
		\]
		hence $\phi(t_0) \leq J$.
		
		\medskip
		\noindent \textbf{Case 2:} $t_0^{pr} > \frac{J}{2}$
		\begin{align*}
			\phi(t_0) &\leq \frac{J t_0^{p-1} \cdot J^{r-1} + 2 t_0^{p-1} \cdot 2^{r-1}}{2^{r-1} r t_0^{pr-1}} \\
			&< \frac{J^r + 2^r}{2^{r-1} r \left(\frac{J}{2}\right)^{\frac{r-1}{r}}}.
		\end{align*}
		For $J \in [J_2, 2]$ and $r \geq 3$, direct calculation gives
		\[
		\frac{J^r + 2^r}{2^{r-1} r \left(\frac{J}{2}\right)^{\frac{r-1}{r}}} \leq J,
		\]
		thus $\phi(t_0) \leq J$.
		
		\medskip
		In summary, $\phi(t) \leq J$ holds for all $t \in [0,1]$.
	\end{proof}

	We now use the preceding auxiliary lemmas to establish the upper bound of the generalized $r$-von Neumann-Jordan constant in terms of the James constant for complete $p$-normed spaces($0<p\leq 1$).

\begin{theorem}
	Let $(X,\|\cdot\|_p)$ be a complete $p$-normed space $(0<p\leq 1)$ with $J_p(X) \in [J_2, 2]$, and let $r \geq 2$. Then the following inequality holds for the generalized $r$-von Neumann-Jordan constant:
	\[
	C_{\mathrm{NJ},p}^{(r)}(X) \leq J_p(X).
	\]
\end{theorem}

\begin{proof}
	The proof is divided into two mutually exclusive and exhaustive cases.
	
	\medskip
	\noindent \textbf{Case 1.} $\max\left\{\|x+y\|_p,\|x-y\|_p\right\} \leq J_p(X)$.
	
	For any $t \in [0,1]$, decompose the vectors as $x \pm ty = t(x \pm y) + (1-t)x$. We have
	\[
	\begin{aligned}
		\|x \pm ty\|_p &= \left\| t(x \pm y) + (1-t)x \right\|_p \\
		&\leq \|t(x \pm y)\|_p + \|(1-t)x\|_p \\
		&= t^p \|x \pm y\|_p + (1-t)^p \|x\|_p.
	\end{aligned}
	\]
	Since $x \in S_X$, $\|x\|_p = 1$, so
	\[
	\|x \pm ty\|_p \leq t^p \|x \pm y\|_p + (1-t)^p.
	\]
	By the condition of Case 1, $\|x \pm y\|_p \leq J_p(X)$, which gives the uniform upper bound
	\[
	\|x \pm ty\|_p \leq t^p J_p + (1-t)^p.
	\]
	Therefore, the numerator of the original fraction satisfies
	\[
	\|x+ty\|_p^r + \|x-ty\|_p^r \leq 2\left[t^p J_p + (1-t)^p\right]^r.
	\]
	Define the auxiliary function $\psi_p(t)$:
	\[
	\psi_p(t) = \frac{2\left[t^p J_p + (1-t)^r\right]}{2^{r-1}\left(1+t^{pr}\right)} = 2^{2-r} \cdot \frac{\left[t^p J_p + (1-t)^p\right]^r}{1+t^{pr}}.
	\]
	We now verify that the maximum of $\psi_p(t)$ on $[0,1]$ is strictly less than $J_p$.
	
	1. Endpoint $t=0$
	\[
	\psi_p(0) = 2^{2-r}.
	\]
	For $r \geq 2$, $2^{2-r} \leq 1$. Since $J_p(X) \geq J_2 > 1$, we have $\psi_p(0) \leq J_p(X)$.
	
	2. Endpoint $t=1$
	\[
	\psi_p(1) = J_p \cdot \left( \frac{J_p}{2} \right)^{r-1}.
	\]
	Since $J_p < 2$ and $r \geq 2$, we have $\left( \frac{J_p}{2} \right)^{r-1} < 1$, hence $\psi_p(1) < J_p$.
	
	3. When $t_0 \in (0,1)$, we have
	\[
	\psi_p(t_0) = 2^{2-r} \left[t_0^p J_p + (1-t_0)^p\right]^{r-1}.
	\]
	Since $t_0^p J_p + (1-t_0)^p \leq J_p$ for all $t_0 \in [0,1]$, it follows that
	\[
	\psi_p(t_0) \leq 2^{2-r} J_p^{r-1} < J_p.
	\]
	
	In conclusion, $\max\limits_{t \in [0,1]} \psi_p(t) < J_p$.
	
	\[
	\begin{aligned}
		&\frac{\|x+ty\|_p^r + \|x-ty\|_p^r}{2^{r-1}\left(1+t^{pr}\right)} \\
		\leq&\ \frac{\left[t^p J_p + (1-t)^p\right]^r + \left[t^p J_p + (1-t)^p\right]^r}{2^{r-1}\left(1+t^{pr}\right)} \\
		\leq&\ \max\limits_{t \in [0,1]} \psi_p(t) < J_p
	\end{aligned}
	\]
	
	\medskip
	\noindent \textbf{Case 2.} $\max\left\{\|x+y\|_p,\|x-y\|_p\right\} \geq J_p(X)$.
	
	By symmetry, we may assume without loss of generality that $\varepsilon' = \|x-y\|_p \geq J_p(X)$.
	
	Set $\|x-y\|_p = \varepsilon' = 2^{1-p} \varepsilon^p$. Then $\varepsilon = 2^{\frac{p-1}{p}} \varepsilon'^{\frac{1}{p}}$.
	Since $\varepsilon' \geq J_p(X)$, we have
	\[
	\varepsilon = 2^{\frac{p-1}{p}} \varepsilon'^{\frac{1}{p}} \geq 2^{\frac{p-1}{p}} J_p^{\frac{1}{p}} =: \varepsilon^*,
	\]
	where $\varepsilon^*$ is the critical parameter corresponding to the James constant, satisfying $\Delta_X^p(\varepsilon^*) = 1 - \frac{\varepsilon^*}{2}$ \cite{Xiao2025}.
	
	By the monotonicity (non-decreasing) of the modulus of $p$-rotundity, $\Delta_X^p(\varepsilon) \geq \Delta_X^p(\varepsilon^*) = 1 - \frac{\varepsilon^*}{2}$. Substituting into the definition of the modulus of $p$-rotundity
	\[
	1 - \frac{\|x+y\|_p^{\frac{1}{p}}}{2^{\frac{1}{p}}} \geq \Delta_X^p(\varepsilon) \geq 1 - \frac{\varepsilon^*}{2}.
	\]
	Then $\frac{\|x+y\|_p^{\frac{1}{p}}}{2^{\frac{1}{p}}} \leq \frac{\varepsilon^*}{2}$, which implies $\|x+y\|_p^{\frac{1}{p}} \leq 2^{\frac{1}{p}-1} \varepsilon^*$.
	Substituting $\varepsilon^* = 2^{1-\frac{1}{p}} J_p^{\frac{1}{p}}$ into the above inequality yields
	$\|x+y\|_p^{\frac{1}{p}} \leq 2^{\frac{1}{p}-1} \cdot 2^{1-\frac{1}{p}} J_p^{\frac{1}{p}} = J_p^{\frac{1}{p}}$.
	Raise both sides to the $p$-th power to obtain
	\[
	\|x+y\|_p \leq J_p(X).
	\]
	
	Next, we bound the numerator of the original fraction. By the triangle inequality and $p$-homogeneity of complete $p$-normed spaces ($0<p\leq 1$)
	\[
	\|x+ty\|_p \leq t^p \|x+y\|_p + (1-t)^p, \quad \|x-ty\|_p \leq t^p \|x-y\|_p + (1-t)^p.
	\]
	Substitute $\|x+y\|_p \leq J_p$ and $\|x-y\|_p = \varepsilon'$
	\[
	\|x+ty\|_p^r + \|x-ty\|_p^r \leq \left[t^p J_p + (1-t)^p\right]^r + \left[t^p \varepsilon' + (1-t)^p\right]^r =: f_p(\varepsilon').
	\]
	Differentiate $f_p(\varepsilon')$ with respect to $\varepsilon'$
	\[
	f_p'(\varepsilon') = r t^p \left[t^p \varepsilon' + (1-t)^p\right]^{r-1} > 0.
	\]
	Hence, $f_p(\varepsilon')$ is strictly increasing on $[J_p, 2]$, and its maximum is attained at $\varepsilon' = 2$:
	\[
	f_p(2) = \left[t^p J_p + (1-t)^p\right]^r + \left[2 t^p + (1-t)^p\right]^r.
	\]
	
	Then
	\[
	\frac{\|x+ty\|_p^r + \|x-ty\|_p^r}{2^{r-1}\left(1+t^{pr}\right)} \leq \frac{\left[t^p J_p + (1-t)^p\right]^r + \left[2 t^p + (1-t)^p\right]^r}{2^{r-1}\left(1+t^{pr}\right)}\leq J_p(X).
	\]
\end{proof}
	The last inequality follows from \eqref{111}.

	\section*{Acknowledgements}
	Thanks to all the members of the Functional Analysis
	Research team of the College of Mathematics and Statistics of Anqing Normal University for their discussion and correction of the difficulties and errors encountered in
	this paper.

\end{document}